\documentclass{tCON2e}

\numberwithin{equation}{section}

\usepackage{amsfonts,amsmath,amssymb}

\usepackage[longnamesfirst,sort]{natbib}% Citation support using natbib.sty
\bibpunct[, ]{(}{)}{;}{a}{,}{,}% Citation support using natbib.sty
\usepackage[longnamesfirst,sort]{natbib}% Citation support using natbib.sty
\bibpunct[, ]{(}{)}{;}{a}{,}{,}% Citation support using natbib.sty
\newtheorem{theorem}{Theorem}[section]

\newtheorem{Th}[theorem]{Theorem}
\newtheorem{Cor}[theorem]{Corollary}
\newtheorem{Lem}[theorem]{Lemma}
\newtheorem{Ex}[theorem]{Example}
\newtheorem{Rem}[theorem]{Remark}

\numberwithin{equation}{section}

\begin{document}

%\jvol{00} \jnum{00} \jyear{} \jmonth{}

\title{Stabilization of Second Order Nonlinear Equations with Variable Delay}

\author{Leonid Berezansky$^{\rm a}$, 
Elena Braverman$^{\rm b}$$^{\ast}$\thanks{$^\ast$Corresponding author. Email: maelena@ucalgary.ca} 
and Lev Idels$^{\rm c}$ \\ \vspace{6pt}
$^{a}${\em{Dept. of Math, Ben-Gurion University of the Negev, Beer-Sheva 84105, Israel}};\\
$^{b}${\em{Dept. of Math \& Stats, Univ. of Calgary, 2500 University Dr. NW, Calgary, AB, Canada T2N1N4}};\\
$^{c}${\em{Dept. of Math, Vancouver Island University (VIU), 900 Fifth St. Nanaimo, BC, Canada  V9S5J5}} \\
\received{received April 2014} }

%\thanks{Partially supported by VIU grant.} \\
%   email:  {\tt{lev.idels@viu.ca}}

%\begin{document}
\maketitle

\begin{abstract}
For a wide class of second order nonlinear non-autonomous  models, we illustrate 
that combining proportional state control with the feedback that is proportional to the derivative of the chaotic signal,
allows to stabilize unstable motions of the system.
The delays are variable, which leads to more flexible controls permitting
delay perturbations; only delay bounds are significant for stabilization by a delayed control.
The results are applied to the sunflower equation which has
an infinite number of equilibrium points.
\end{abstract}


%\begin{classcode} 
%93D15, 34K20, 34K35, 34H15, 93B05
%\end{classcode}\bigskip

\section{Introduction}

It is well known that control of dynamical systems is a classical subject in engineering sciences.
Time delayed feedback control is an efficient method for stabilizing unstable periodic orbits of
chaotic systems which are described by second order delay differential equations, see 
\cite{Bo,Do,Fr1,Fr2,Kim,Liu,Reit, Gabor2,Gabor1,Wan}. 
When introducing a control, we assume that the chosen equilibrium of an equation is unstable, 
and the controller will transform the unstable equation into an asymptotically or exponentially stable equation.
Instability tests for some autonomous delay models of the second order could be found, for example,
in \cite{Cah}.
Two basic proportional (adaptive) control models are widely used: 
standard feedback controllers $u(t)=K[x(t)-x^{\ast}]$ with
the controlling force proportional to the deviation of the system from the attractor,
where $x^{\ast}$ is an equilibrium of the equation,
and the delayed feedback control $u(t)=K[x(t-\tau(t))-x(t)]$, see
\cite{Boccal,Joh,Ko}.

Proportional control fails if there exist rapid changes to the system that come from an external source,
and to keep the system steady under an abrupt change,  a derivative control was used in \cite{ Bela, Reit, Vyh}, i.e.
$u(t)=\beta \frac{d}{dt}e(t)$,
where, for example, $e(t)=x(t-\tau)-x(t)$ or $e(t)=x(t)-x^{\ast}$.
In electronics, a simple operational amplifier differentiator
circuit will generate the continuous feedback signal which is proportional to the time derivative
of the voltage across the negative resistance, see \cite{Joh}.
A classical proportional control  does not stabilize
even linear ordinary differential equations; e.g. the equation $\ddot{x}=u(t)$ with the control
$u(t)=K[x(t-\tau(t))-x(t)]$ is not asymptotically stable for any $K$, since any constant is a solution of
this equation. The pure derivative control $u(t)=-\lambda\dot{x}(t)$ also does not stabilize all second order differential equations.
For example, the equation $\ddot{x}+ax(t)-ax(t-\tau)=u(t)$ with the control
$u(t)=-\lambda\dot{x}(t)$ is not asymptotically stable for any control since any constant is a solution of this equation.
Some interesting and novel results could be found in \cite{Ren,Ru,Si,Saberi2,Yan}.
For a linear non-autonomous model
$\dot{x}=A(t)x(t)+B(t) u(t)$ the effective multiple-derivative feedback controller 
$\displaystyle u(t)=\sum_{i=0}^{M-1} K_{i}x^{(i)}(t-h)$ was introduced in \cite{Saberi1},
and a special transformation was used to transform neutral-type DDE into a retarded DDE.
However, most of second order applied models are nonlinear, even the original pendulum equation.
The main focus of the paper is the control of nonlinear delay equations, some real world models are considered
in Examples~\ref{ex2},\ref{ex3},\ref{additional_example}.

In the present paper we study  a nonlinear second order delay differential equation 
\begin{equation}\label{3}
\ddot{x}(t)+\sum_{k=1}^l f_k(t,\dot{x}(g_k(t)))+\sum_{k=1}^m s_k(t,x(h_k(t)))=u(t),~~t \geq t_0,
\end{equation}
with the input or the controller $u(t)$, along with its linear version
\begin{equation}\label{1}
\ddot{x}(t)+\sum_{k=1}^l a_k(t)\dot{x}(g_k(t))
+\sum_{k=1}^m b_k(t)x(h_k(t))=u(t), ~~t \geq t_0.
\end{equation}
Both equations (\ref{1}) and (\ref{3}) satisfy for any $t_0\geq 0$ the initial condition
\begin{equation}\label{2}
x(t)=\varphi(t), ~\dot{x}(t)=\psi(t), ~t\leq t_0.
\end{equation}
We will assume that the initial value problem has a unique global solution on $[t_0,\infty)$
for all nonlinear equations considered in this paper, and
the following conditions are satisfied:
\\
(a1) $a_i, b_j$ are Lebesgue measurable and essentially bounded on $[0,\infty)$ functions,
$i=1,\dots,l$, $j=1,\dots,m $, which allows to define essential eventual limits
\begin{equation}\label{6}
\alpha=\limsup_{t\rightarrow\infty}\sum_{k=1}^l |a_k(t)|,~\beta=\limsup_{t\rightarrow\infty}\sum_{k=1}^m |b_k(t)|;
\end{equation}
(a2) $h_j, g_i$ are  Lebesgue measurable functions,
$
h_i(t)\leq t, g_i(t)\leq t,$ $\lim\limits_{t\to\infty} h_i(t)=\infty$, $\lim\limits_{t\to\infty}
g_i(t)=\infty$, $i=1,\dots,l$, $j=1,\dots,m $.

The paper is organized as follows.
In Section 2 we design a stabilizing damping control
$u(t)=\lambda_1 \dot{x}(t)+\lambda_2 (x(t)-x^{*})$ {\em for any linear non-autonomous equation} (\ref{1}). 
Under some additional condition on the functions $f_k$ and $s_k$, such control also stabilizes equations of type (\ref{3}). The 
results are based on stability tests recently obtained in  \cite{BBD, BBI}
for second order non-autonomous differential equations. 
We also prove in Section 2 that a strong enough controlling force, depending on the derivative and the present 
(and past) positions, can globally stabilize an equilibrium of the controlled equation. 
In Section 3 classical proportional delayed feedback controller $u(t)=K[x(t-\tau(t))-x(t)]$ is applied 
to stabilize a certain class of second order delay equations with a single delay involved in the state term only. 
We develop tailored feedback controllers and justify their application both analytically and numerically.

\section{Damping Control}

We will use auxiliary results recently obtained in \cite{BBD,BBI}.

\begin{Lem}\label{lemma1}\cite[Corollary 3.2]{BBD}
If $a>0, b>0,$
\begin{equation}\label{5}
4b>a^2, ~~\frac{2(a+\sqrt{4b-a^2})}{a\sqrt{4b-a^2}}\alpha+
\frac{4}{a\sqrt{4b-a^2}}\beta<1,
\end{equation}
where $\alpha$ and $\beta$ are defined in (a1) by (\ref{6}),
then the zero solution of the equation 
\begin{equation}\label{4}
\ddot{x}(t)+a\dot{x}(t)+bx(t)+\sum_{k=1}^l a_k(t)\dot{x}(g_k(t))
+\sum_{k=1}^m b_k(t)x(h_k(t))=0
\end{equation}
is globally exponentially stable.
\end{Lem}

\begin{Lem}\label{lemma2}\cite{BBI}
Assume that the equation
\begin{equation}\label{19}
\ddot{x}(t)+f(t,x(t),\dot{x}(t))+s(t,x(t))+\sum_{k=1}^m s_k(t,x(t),x(h_k(t)))=0
\end{equation}
possesses a unique trivial equilibrium, where
$f(t,v,0)=0$, $s(t,0)=0$, $s_k(t,v,0)=0$, \\ $0<a_0\leq \frac{f(t,v,u)}{u}\leq A$,
$\displaystyle 0<b_0\leq \frac{s(t,u)}{u}\leq B$, 
$\displaystyle \left| \frac{s_k(t,v,u)}{u} \right| \leq C_k, u\neq 0,~t-h_k(t)\leq \tau$.

If at least one of the conditions \\
1) $\displaystyle B\leq \frac{a_0^2}{4},~ \sum_{k=1}^m C_k<b_0-\frac{a_0}{2}(A-a_0)$, ~~~~~~
2) $\displaystyle b_0\geq \frac{a_0}{2}\left(A-\frac{a_0}{2}\right),~ \sum_{k=1}^m C_k<\frac{a_0^2}{2}-B$\\
holds, then zero is a global  attractor for all solutions of equation~(\ref{19}).
\end{Lem}

We start with linear equations.
Stabilization results for linear systems were recently obtained in  \cite{Saberi1,Saberi2}.
Unlike \cite{Saberi1,Saberi2}, the following theorem considers models with variable delays, however,
the control is not delayed.

\begin{Th}\label{th1}\label{cor1}
For any  $\delta\in (0,2)$, $\alpha$ and $\beta$ defined by (\ref{6}) and
\begin{equation}\label{8}
\lambda > \mu(\lambda) := \frac{(\delta+\sqrt{4-\delta^2}) \alpha+\sqrt{(\delta+\sqrt{4-\delta^2})^2\alpha^2
+4\sqrt{4-\delta^2}\beta\delta}}{\delta \sqrt{4-\delta^2}}\, ,
\end{equation}
equation (\ref{1}) with the control $u(t)=-\delta \lambda \dot{x}(t)-\lambda^2 x(t)$
is exponentially stable.
\end{Th}
\begin{proof}
Equation (\ref{1}) with the control
\begin{equation}\label{7a}
\ddot{x}(t)+\sum_{k=1}^l a_k(t)\dot{x}(g_k(t))
+\sum_{k=1}^m b_k(t)x(h_k(t))=-\delta \lambda\dot{x}(t)-\lambda^2 x(t)
\end{equation}
has the form of (\ref{4}) with $a=\delta\lambda$ and $b=\lambda^2$.
Then the inequalities in (\ref{5}) have the form
\begin{equation}\label{9}
4\lambda^2>\delta^2\lambda^2 \mbox{~~and~~}
\frac{2(\delta+\sqrt{4-\delta^2})}{\delta\lambda\sqrt{4-\delta^2}}\alpha+\frac{4}{\delta\lambda^2\sqrt{4-\delta^2}}\beta<1.
\end{equation}

The first inequality in (\ref{9}) holds as $\delta\in (0,2)$,
and the second one is equivalent to
\begin{equation}\label{10}
\delta\lambda^2\sqrt{4-\delta^2} -2 \left( \delta+\sqrt{4-\delta^2} \right) \alpha\lambda-4\beta>0.
\end{equation}
Condition  (\ref{8}) implies   (\ref{10}), which completes the proof.
\end{proof}

\begin{Cor}\label{cor2a}
Let $\displaystyle \mu(\delta_0)=\min_{\delta\in [\varepsilon,2-\varepsilon]} \mu(\delta)$ for some $\varepsilon>0$,
where $\mu(\delta)$ is defined in (\ref{8}).
Then for $\lambda > \mu(\delta_0)$ equation (\ref{1}) with the control  $u(t)=-\delta_0 \lambda\dot{x}(t)-\lambda^2 x(t)$
is exponentially stable.
\end{Cor}

For $\delta=\sqrt{2}$
Theorem~\ref{th1} yields the following result.
\begin{Cor}\label{cor2}
Eq. (\ref{1}) with the control
$u(t)=-\sqrt{2}\lambda \dot{x}(t)-\lambda^2 x(t)$
is exponentially stable if
\begin{equation}\label{12a}
\lambda>\sqrt{2}(\alpha+\sqrt{\alpha^2+\beta}).
\end{equation}
\end{Cor}

\begin{Rem} For any equation (\ref{1}) there exists $\lambda>0$ such that condition (\ref{12a}) holds.
Hence the stabilizing damping control exists for any equation of form (\ref{1}).
\end{Rem}
\begin{Ex}\label{ex1} For the equation
\begin{equation}\label{11}
\ddot{x}(t)+ (\sin t) \dot{x}(g(t))+ (\cos t) x(h(t))=0, ~~h(t)\leq t, ~ g(t)\leq t,
\end{equation}
the upper bounds defined in (\ref{6}) are $\alpha=\beta=1$.
Hence, as long as $\lambda>2+\sqrt{2}$ in Corollary~\ref{cor2}, equation (\ref{11}) with the  control
$u(t)=-\sqrt{2}\lambda \dot{x}(t)-\lambda^2 x(t)$
is exponentially stable.
\end{Ex}
Let us proceed to nonlinear equation (\ref{3}); its stabilization is the main object of the present paper.
For simplicity we consider here nonlinear equations with the zero equilibrium, since the change of the variable
$z=x-x^{\ast}$  transforms an equation with the equilibrium $x^*$ into an equation in $z$ with the zero
equilibrium.

\begin{Th}\label{th2}
Suppose $f_k(t,0)=s_k(t,0)=0$,
\begin{equation}\label{13a}
\displaystyle \left|\frac{f_k(t,u)}{u}\right|\leq a_k(t),~\left|\frac{s_k(t,u)}{u}\right|\leq b_k(t), u\neq 0.
\end{equation}
Then for any $\delta\in (0,2)$, the zero equilibrium of (\ref{3})
with the control $u(t)=-\delta\lambda\dot{x}(t)-\lambda^2 x(t)$
\begin{equation}\label{12}
\ddot{x}(t)+\sum_{k=1}^l f_k(t,\dot{x}(g_k(t)))+\sum_{k=1}^m s_k(t,x(h_k(t)))=-\delta\lambda\dot{x}(t)-\lambda^2 x(t)
\end{equation}
is globally asymptotically stable, provided (\ref{8}) holds with $\alpha$ and $\beta$ defined in (\ref{6}).
\end{Th}
\begin{proof}
Suppose $x$ is a fixed solution of equation ~(\ref{12}).
Equation~(\ref{12}) can be rewritten as
$$
\ddot{x}(t)+\sum_{k=1}^l a_k(t)\dot{x}(g_k(t))
+\sum_{k=1}^m b_k(t)x(h_k(t))=-\delta\lambda\dot{x}(t)-\lambda^2 x(t),
$$
where
$\displaystyle
a_k(t)= \left\{\begin{array}{cc}
\frac{f_k(t,\dot{x}(t))}{\dot{x}(t)},& \dot{x}(t)\neq 0,\\
0,& \dot{x}(t)=0,\end{array}\right. ~~
b_k(t)=\left\{\begin{array}{cc}
\frac{s_k(t,x(t))}{x(t)},& x(t)\neq 0,\\
0,& x(t)=0. \end{array}\right.
$
~~~Hence the function $x$ is a solution of the linear equation
\begin{equation}\label{13}
\ddot{y}(t)+\sum_{k=1}^l a_k(t)\dot{y}(g_k(t))
+\sum_{k=1}^m b_k(t)y(h_k(t))=-\delta\lambda\dot{y}(t)-\lambda^2 y(t),
\end{equation}
which is exponentially stable by  Theorem~\ref{th1}.
Thus $\lim\limits_{t\rightarrow\infty}y(t)=0$ for any solution
$y$ of equation (\ref{13}), and since $x$ is a solution of (\ref{13}),
$\lim\limits_{t\rightarrow\infty}x(t)=0.$
\end{proof}

In particular, for $\delta=\sqrt{2}$ condition (\ref{8}) transforms into (\ref{12a}).

\begin{Ex}\label{ex2} Consider the equation
\begin{equation}\label{14}
\ddot{x}(t)+a(t) \dot{x}(g(t))+ b(t)\sin (x(h(t)))=0, ~~h(t)\leq t, ~g(t)\leq t,
\end{equation}
with $|a(t)|\leq \alpha, |b(t)|\leq \beta$.
Equation (\ref{14}) generalizes the sunflower equation introduced by Israelson and Johnson in \cite{isa} as a model
for the geotropic circumnutations of {\em Helianthus annuus}; later it was studied in \cite{alf,liza,somolinos}.

We have $\displaystyle \left|\frac{\sin u}{u}\right|\leq 1, ~u\neq 0$; hence if condition (\ref{12a}) holds
for $\alpha=\limsup\limits_{t\rightarrow\infty} |a(t)|$ and $\beta=\limsup\limits_{t\rightarrow\infty} |b(t)|$,
then the zero equilibrium of equation (\ref{14}) with the control
$u(t)=-\sqrt{2}\lambda \dot{x}(t)-\lambda^2 x(t)$ in the right-hand side
is globally exponentially stable.
Equation (\ref{14}) has an infinite number of equilibrium points
$x^{\ast}=\pi k$, $k=0,1,\dots$. To stabilize a fixed equilibrium $x^{\ast}=\pi k$
we apply the controller $u(t)=-\sqrt{2}\lambda \dot{x}(t)-\lambda^2 (x(t)-x^{\ast})$.

For example, consider the sunflower equation
$$
\ddot{x}(t)+ \dot{x}(t)+ 2 \sin (x(t-\pi))=0
$$
with various initial conditions $x(0)=6, 3, 0.1$, where $x(t)$ is constant for $t\leq 0$,
$x'(0)=1$ which has chaotic solutions (see Fig.~\ref{figure2}, left).
Application of the controller $u(t)=-\lambda \delta\dot{x}(t)- \lambda^2  [x(t)-\pi]$,
where $\delta=\sqrt{2}$ and $\lambda > \sqrt{2}+\sqrt{6}$, for example, $\lambda=4$,
stabilizes the otherwise unstable equilibrium $x^{\ast}=\pi$, as illustrated in Fig.~\ref{figure2}, right.

\begin{figure}[ht]
\centering
\vspace{-15mm}
\includegraphics[scale=0.38]{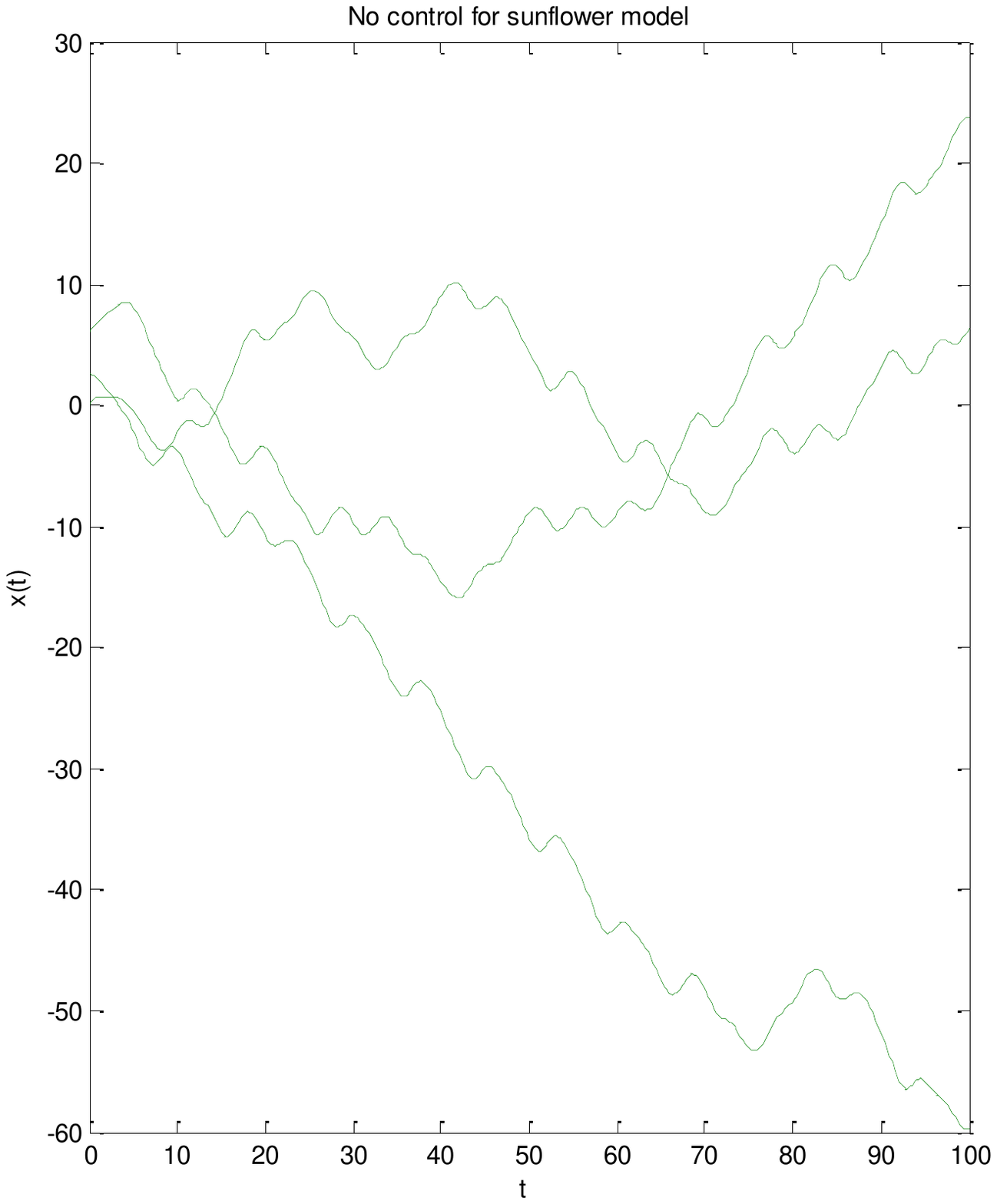} \hspace{-6mm}
\includegraphics[scale=0.38]{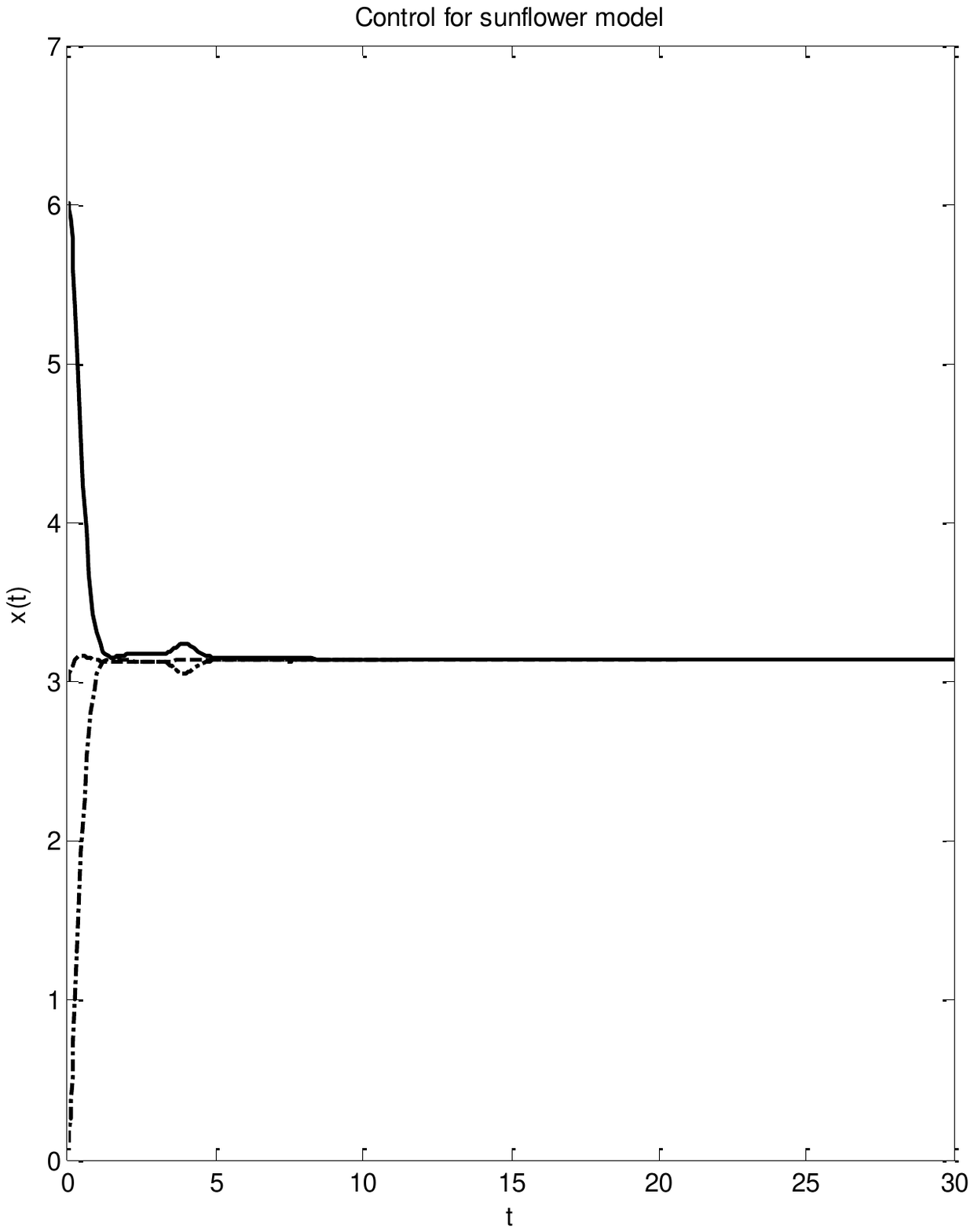}
\vspace{-28mm}
\caption{Stabilization of the equilibrium of the sunflower equation
$\ddot{x}(t)+\dot{x}+2 \sin(x(t-\pi))=0$ with various initial conditions.
The left graph illustrates unstable (chaotic) solutions while in the right graph,
corresponding to the sunflower model  with the control $u(t)=-\lambda \delta\dot{x}(t)- \lambda^2  [x(t)-\pi]$,
all three solutions of the controlled equation converge to the equilibrium $\pi$.
}
\label{figure2}
\end{figure}
\end{Ex}

\section{Classical proportional control}

In this section we investigate stabilization with the standard proportional delayed control.

Consider the equation
\begin{equation}\label{1a}
\ddot{x}(t)+a\dot{x}(t)+bx(h(t))=f(t,x(g(t)))
\end{equation}
which has an  equilibrium $x^{\ast}$.
The equation
\begin{equation}
\label{1ab}
\ddot{x}(t)+a\dot{x}(t)+bx(h(t))=f(t,x(g(t)))+u(t)
\end{equation}
with the control $u(t)=-b[x(t)-x(h(t))]$ has 
%$\ddot{x}(t)+a\dot{x}(t)+bx(h(t))=f(t,x(g(t)))-bx(t)+ bx(h(t))$.
the same equilibrium as (\ref{1ab}) 
and can be rewritten as 
\begin{equation}\label{20}
\ddot{x}(t)+a\dot{x}(t)+bx(t)=f(t,x(g(t))).
\end{equation}
After the substitution
$x=y+x^{\ast}$ into equation (\ref{20}) we obtain
\begin{equation}\label{21}
\ddot{y}(t)+a\dot{y}(t)+by(t)=p(t,y(g(t))),
\end{equation}
with $p(t,v)=f(t,v+x^{\ast})-bx^{\ast}$, where (\ref{21}) has the zero equilibrium.

\begin{Th}\label{th3}
Suppose $\left|f(t,v+x^{\ast})-bx^{\ast} \right|\leq C |v|$ for any $t$ and at least one of the following conditions
\\
a) $C<b\leq a^2/4$; ~~
b) $a^2/4\leq b <a^2/2-C$; ~~
c) $C< a\sqrt{4b-a^2}/4$
\\
holds.
Then the equilibrium $x^{\ast}$ of equation (\ref{1a})
 with the control $u(t)=-b[x(t)-x(h(t))]$ is globally
asymptotically stable.
\end{Th}
\begin{proof}
Statements a) and b) of Theorem~\ref{th3} are  direct corollaries of Lemma~\ref{lemma2}.
To prove Part c) suppose that $x$ is a solution of equation (\ref{21}). Equation (\ref{21}) can be rewritten in the form
\begin{equation}\label{21a}
\ddot{x}(t)+a\dot{x}(t)+bx(t)=P(t)x(g(t)),
\end{equation}
where ~
$\displaystyle
P(t)=\left\{\begin{array}{cc}
\frac{p(t,x(t))}{x(t)},& x(t)\neq 0,\\
0,& x(t)=0. \end{array}\right.
$~~
Hence the function $x$ is a solution of the linear equation
\begin{equation}\label{21b}
\ddot{y}(t)+a\dot{y}(t)+by(t)=P(t)y(g(t)).
\end{equation}
If $\alpha=0, \beta =C$, then condition c) of the theorem coincides with condition
(\ref{5}) of Lemma~\ref{lemma1}. Hence by Lemma~\ref{lemma1} equation (\ref{21b})
is exponentially stable, i.e.
for any solution $y$ of this equation
we have $\limsup\limits_{t\rightarrow\infty} y(t)=0$. Hence for a fixed solution $x$
of equation (\ref{21a})  we also have  $\limsup\limits_{t\rightarrow\infty} x(t)=0$.
\end{proof}

Let us examine a popular model
\begin{equation}\label{21abcd}
\ddot{x}(t)+a\dot{x}(t)+bx(h(t))=F(x(g(t)),
\end{equation}
where $F$ is either monotone or non-monotone feedback.
Its applications include the
neuromuscular regulation of movement and posture,
acousto-optical bistability, metal cutting, the cascade
control of fluid level devices and the electronically
clamped pupil light \cite{Cam}.

\begin{Ex}
\label{ex3}
Consider the special case of (\ref{21abcd}) 
\begin{equation}\label{1A}
\ddot{x}(t)+a\dot{x}(t)+bx(h(t))=\frac{d(t)|x(g(t))|^{m+1}}{1+|x(g(t))|^n},
\end{equation}
where $ 0\leq m<n, |d(t)|\leq d_0.$
Denote
~~ $\displaystyle
\mu=\sup_{v\geq 0}\frac{v^{m}}{1+v^n}=\left\{\begin{array}{ll}
1,& m=0,\\
\displaystyle \frac{m^{m/n}}{n(n-m)^{m/n-1}},& m>0.
\end{array}\right.
$ \\
If the conditions of Theorem~\ref{th3} hold with $C=\mu d_0$ then the zero equilibrium of
equation (\ref{1A})  with the control $u(t)=-b[x(t)-x(h(t))]$ is globally asymptotically  stable.
\end{Ex}

\begin{Ex}
\label{additional_example}
Consider the particular case of (\ref{21abcd})
\begin{equation}\label{1abc}
\ddot{x}(t)+ 2 \dot{x}(t)+  x(h(t))=\frac{0.8 x(g(t))}{1+x^{n}(g(t))},
\end{equation}
where $n \geq 6$. As can be easily verified, the range of the function $f(x)=1.6 x/(1+x^n)$ includes $[-1,1]$ for $n \geq 6$.
Let us demonstrate that for a certain choice of $h(t)$ and $g(t)$ the function $x(t) =\sin(t/4)$
is a solution. We restrict ourselves to the interval $[0,8\pi]$, and then extend it in such a way that
both $t-h(t)$ and $t-g(t)$ are periodic with a period $8\pi$. We can find $h(t)\in [0,t]$ such that
$\sin(h(t)/4)= \frac{1}{16} \sin(t/4)$, since $\sin(0)=0$, and the continuous function takes all its values between zero and $\sin(t/4)$.
As mentioned above, the function $y=1.6 u/(1+u^{n})$ takes all the values $y\in [-1,1]$ for $u\in [-1,1]$,
and $\cos(x/4)$ takes all the values between -1 and 1 for $x\in [-4\pi, t]$, there is $g(t)$ such that
$\displaystyle \frac{1}{2} \cos \left( \frac{t}{4} \right)=\frac{0.8 \sin(g(t)/4)}{1+\sin^{n}(g(t)/4)}$ and $g(t) \in [-4\pi, t]$.
Then $x(t) =\sin(t/4)$ is a solution of (\ref{1abc}) on $[0,4\pi]$, with the same initial function on $[-4\pi,0]$.
Further, we extend $h(t+8\pi)=h(t)+8\pi$, $g(t+8\pi)=g(t)+8\pi$ and obtain that $x(t) =\sin(t/4)$ is a solution of (\ref{1abc}),
$t \geq 0$, with $\varphi(t)=\sin(t/4)$, $t\leq 0$, and a bounded (by $16\pi$) delay.
Hence, equation (\ref{1abc}) is not asymptotically stable.

Equation  (\ref{1abc}) with control $u=-(x(t)-x(h(t)))$ becomes ~~
$\displaystyle
\ddot{x}(t)+ 2 \dot{x}(t)+  x(t)=\frac{0.8 x(g(t))}{1+x^n(g(t))}$, 
and it is globally asymptotically  stable by  Theorem~\ref{th3}, Part a), since $C=0.8<b=1=a^2/4$.
Fig.~\ref{figure3} numerically illustrates the results for the constant delay $g(t)=t-\tau$.
\end{Ex}

%The results of Examples \ref{additional_example} are numerically illustrated in Fig.~\ref{figure3}.

\begin{figure}[ht]
\centering
\vspace{-15mm}
\includegraphics[scale=0.38]{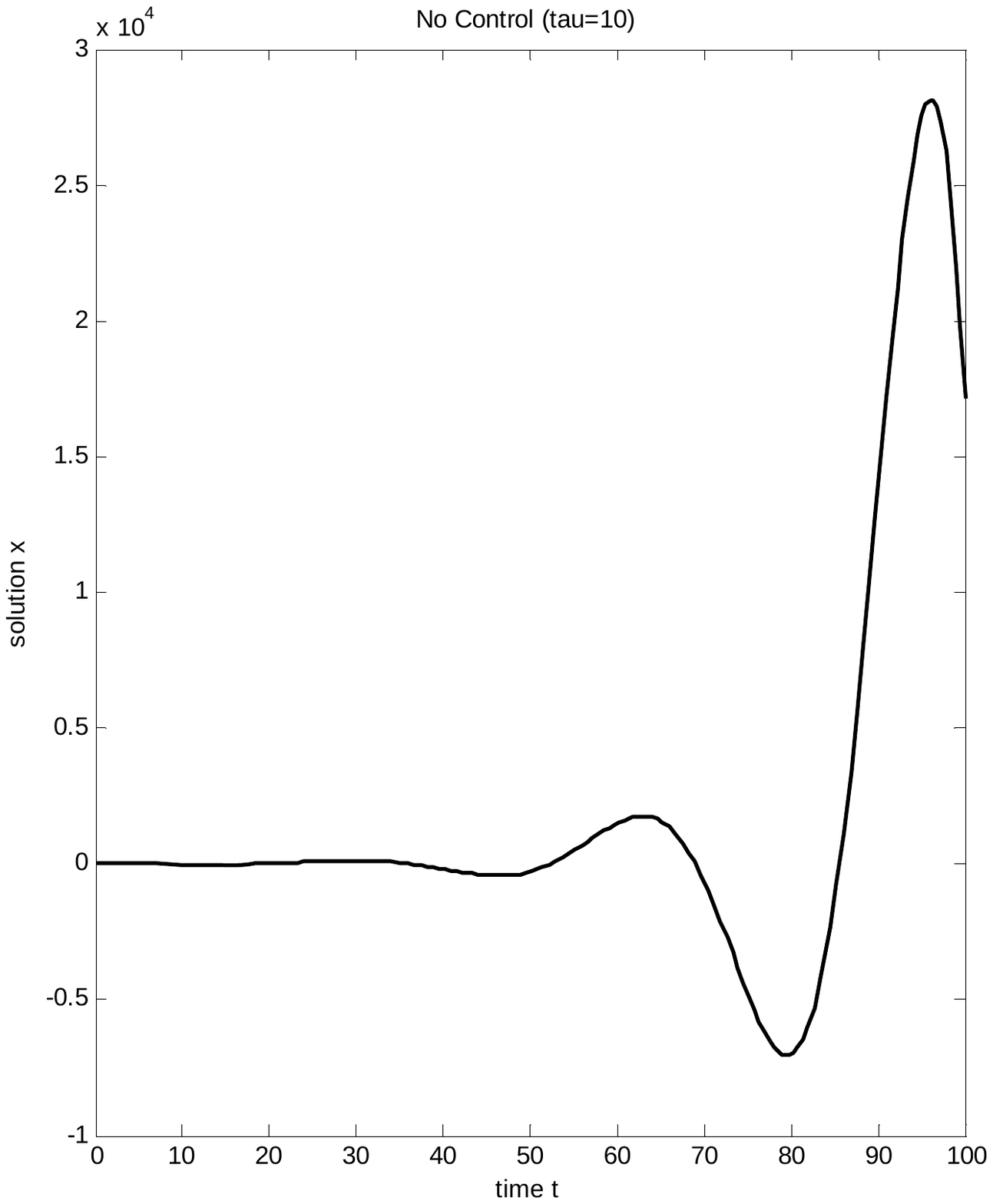} \hspace{-6mm}
\includegraphics[scale=0.38]{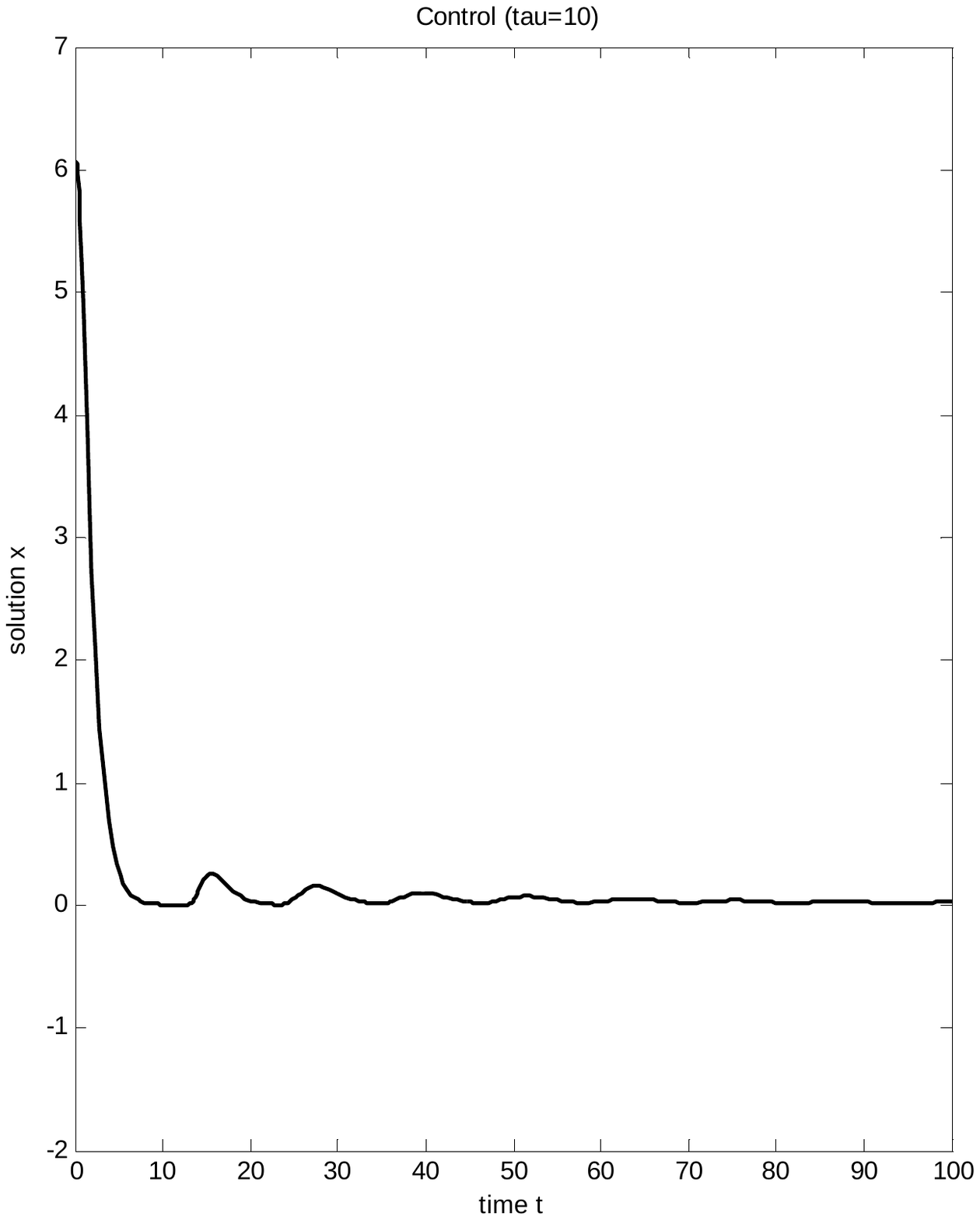}
\vspace{-28mm}
\caption{Stabilization with a proportional control for the equation
$\ddot{x}(t)+2\dot{x}+x(t-\tau))=\frac{0.8 x(t-\tau)}{1+x^{8}(t-\tau)}$ here $\tau=10$ .
The left graph illustrates an unstable (oscillating and unbounded) solution while in the right graph,
the control $u(t)=x(t-\tau)-x(t)$ produces a stable trajectory.
}
\label{figure3}
\end{figure}
Consider the nonlinear equation
\begin{equation}\label{2a}
\ddot{x}(t)+a\dot{x}(t)+f(t,x(h(t)))=0
\end{equation}
which has an equilibrium $x(t)=x^{\ast}$.
For stabilization we will use the controller $u=-K[x(t)-x^{\ast}]$, $K>0$
and obtain the equation
\begin{equation}\label{3a}
\ddot{x}(t)+a\dot{x}(t)+f(t,x(h(t)))=-K[x(t)-x^{\ast}].
\end{equation}
The substitution of $y(t)=x(t)-x^{\ast}$ into equation (\ref{3a}) yields
\begin{equation}\label{4a}
\ddot{y}(t)+a\dot{y}(t)+p(t,y(h(t)))=-Ky,
\end{equation}
where $p(t,v)=f(t,v+x^{\ast})$.

\begin{theorem}\label{th4}
Suppose $\displaystyle \left| f(t,v+x^{\ast}) \right| \leq C |v|$,  and
at least one of the conditions holds:
\\
a) $C<K\leq a^2/4$;
~~
b)  $a^2/4\leq K<a^2/2-C$;
~~
c)  $ C< a\sqrt{4K-a^2}/4$.

Then the equilibrium $x^{\ast}$ of equation (\ref{2a}) with the control
$u=-K(x(t)-x^{\ast})$ is globally asymptotically  stable.
\end{theorem}
\begin{proof}
Equation (\ref{4a}) has the form
$
\ddot{y}(t)+a\dot{y}(t)+Ky=-p(t,y(h(t))),
$
and application of Lemmas~\ref{lemma1} and \ref{lemma2} concludes the proof.
\end{proof}

To illustrate application of Theorem \ref{th4}, consider the sunflower equation
\begin{equation}\label{5a}
\ddot{x}(t)+a\dot{x}(t)+A\sin (\omega x(h(t)))=0,~~ a,A,\omega>0.
\end{equation}
This equation has an infinite number of unstable equilibrium points $x=\frac{(2k+1)\pi}{\omega}$, $k=0,1,\dots$, see \cite{BBI}.
To stabilize a fixed equilibrium $x^*=\frac{(2k+1)\pi}{\omega}$ of equation \eqref{5a}, we choose the
controller $u=-K\left[x(t)-\frac{(2k+1)\pi}{\omega}\right], K>0$, i.e.
\begin{equation}\label{6a}
\ddot{x}(t)+a\dot{x}(t)+A\sin (\omega x(h(t)))=-K\left[x(t)-\frac{(2k+1)\pi}{\omega}\right].
\end{equation}
Since $\displaystyle \left|A\sin (\omega v)\right|\leq A\omega |v|$, Theorem~\ref{th4} implies the following result.

\begin{Cor}\label{cor5}
Suppose at least one of the conditions holds:
\\
a) $A\omega<K\leq a^2/4$;
~~
b)  $a^2/4\leq K<a^2/2-A\omega$;
~~
c)  $A\omega<  a \sqrt{4K-a^2}/4$.

Then the equilibrium $x^{\ast}=\frac{(2k+1)\pi}{\omega}$ of equation (\ref{6a})
is globally asymptotically  stable.
\end{Cor}

\section{Summary}

The results of the paper can be summarized as follows:
\begin{enumerate}
\item
For a wide class of nonlinear delay second order equations,
we developed stabilizing controls combining the proportional feedback with
the proportional derivative feedback.
\item
We designed a standard feedback controller which allows to stabilize a second order nonlinear
equation with a linear nondelay damping term.
\end{enumerate}

The results are illustrated using nonlinear models with several equilibrium points, for example, modifications of the
sunflower equation.

\subsection{Acknowledgments}

L. Berezansky was partially supported by Israeli Ministry of Absorption,
E. Braverman was partially supported by the NSERC DG Research Grant,
L. Idels was partially supported by a grant from VIU.
The authors are very grateful to the editor and the anonymous referees whose comments significantly contributed to the 
better presentation of the 
results of the paper.

\end{document}